\documentclass[11pt, oneside]{article}   	
\usepackage{geometry}               		
\usepackage{graphicx}				
\usepackage{amssymb}
\usepackage{fancyhdr}
\usepackage{amsmath}
\usepackage{mathtools}
\usepackage{changepage}
\usepackage{dsfont}
\usepackage{bm}
\usepackage{amsthm}
\usepackage{subcaption}
\usepackage{float}
\usepackage{color}
\usepackage{enumerate} 
\usepackage{authblk}
\usepackage{braket}
\usepackage{algorithm}
\usepackage{algpseudocode}
\usepackage{hyperref}

\usepackage{tikz}
\usetikzlibrary{shapes}
\usetikzlibrary{topaths}
\usetikzlibrary{patterns}
\usetikzlibrary{decorations.pathmorphing}

\def\ve#1{\mathchoice{\mbox{\boldmath$\displaystyle\bf#1$}}
{\mbox{\boldmath$\textstyle\bf#1$}}
{\mbox{\boldmath$\scriptstyle\bf#1$}}
{\mbox{\boldmath$\scriptscriptstyle\bf#1$}}}

\DeclareMathOperator{\supp}{supp}
\newcommand{\R}{\mathds{R}}
\newcommand{\Z}{\mathds{Z}}
\newcommand{\pol}{\{\bm{x}\in\R^n:A\bm{x}=\bm{b},\, B\bm{x}\leq \bm{d}\}}

\newcommand{\mc}{\mathcal{C}}
\newcommand{\B}{\mathsf{B}}
\newcommand{\me}{\mathcal{E}}
\newcommand{\mcp}{\mathcal{P}}

\newcommand\vecc{{\ve c}}

\newcommand\vex{{\ve x}}

\DeclareMathOperator{\fmatOp}{FMAT}
\newcommand{\fmat}{{\fmatOp}}

\newcommand{\T}{{\intercal}}

\newtheorem{lem}{Lemma}
\newtheorem{claim}{Claim}
\newtheorem{thm}{Theorem}
\newtheorem{cor}{Corollary}
\newcommand{\C}{\mathcal{C}}
\newtheorem{defn}{Definition}

\newtheorem{prop}{Proposition}

\DeclareMathOperator{\lcm}{lcm}

\DeclareMathOperator{\rk}{rk}

\title{Pivot Rules for Circuit-Augmentation \\
Algorithms  in Linear Optimization}

 \author{Jes\'us A. De Loera$^{*}$, Sean Kafer$^{\star}$, Laura Sanit\`a$^{\star}$}

\date{}							

 \affil{$^{*}$ University of California, Davis, USA. \texttt{deloera@math.ucdavis.edu}\\
 $^{\star}$ Department of Combinatorics and Optimization, University of Waterloo, 
 Canada. \texttt{\{skafer, lsanita\}@uwaterloo.ca}}

\begin{document}
\maketitle

\begin{abstract}
Circuit-augmentation algorithms are generalizations of the Simplex method,
where in each step one is allowed to move along a fixed set of directions, called circuits, 
that is a superset of the edges of a polytope. We show that in the circuit-augmentation 
framework the greatest-improvement and Dantzig pivot rules are NP-hard, already for 0/1-LPs. 
 Differently, the steepest-descent pivot rule can be carried out in polynomial time in the 0/1 setting, 
 and the number of circuit augmentations required to reach an optimal solution according to this rule 
 is strongly-polynomial for 0/1-LPs. 
 
 The number of circuit augmentations has been of interest
 as a proxy for  the number of steps in the Simplex method, and the circuit-diameter of polyhedra 
 has been studied as a lower bound to the combinatorial diameter of polyhedra. Extending prior results, 
 we show that for any polyhedron $P$ the circuit-diameter is bounded by a polynomial in the input bit-size of $P$. 
 This is in contrast with the best bounds for the combinatorial diameter of polyhedra.

Interestingly, we show that the circuit-augmentation framework can be exploited  to make novel conclusions 
about the classical Simplex method itself: In particular, as a byproduct of our circuit results, we prove that 
(i) computing the shortest (monotone) path to an optimal solution on the 1-skeleton of a polytope is NP-hard, 
and hard to approximate within a factor better than 2, and (ii) for $0/1$ polytopes, a monotone path of 
strongly-polynomial length can be constructed using steepest improving edges.
\end{abstract}

\section{Introduction}
Linear Programming (LP) is one of the most powerful mathematical tools for tackling optimization problems.
While various algorithms have been proposed for solving LPs in the past decades, probably the most 
popular method remains the \emph{Simplex method}, introduced by G. B. Dantzig in the 1940's. In what 
follows we assume the input LP is given in the general minimization format 

\begin{equation}\label{lp}
\min\set{\,\bm c^\T\bm x\ :\, A \bm x=\bm b,\, B \bm x\leq \bm d, \, \bm x\in\R^n\,}.
\end{equation}

Here the objective function is minimized over the polyhedron $\mcp=\pol$ of feasible solutions. 
The $1$-skeleton of the polyhedron $\mcp$ is the graph given by the $0$-dimensional faces (vertices) 
and $1$-dimensional faces (edges) of $\mcp$. Given an objective function to minimize on $\mcp$, a path 
on the 1-skeleton is called monotone if every vertex on the path has an associated objective function value 
larger than its subsequent one. The Simplex method starts with  an initial feasible extreme point solution of 
the LP, and in each step, it moves along an improving edge to an adjacent extreme point, until an optimal solution 
is found or unboundedness is detected. Effectively there is a \emph{monotone} path from an initial extreme 
point solution of an LP to an optimal one. 

Despite having been used and studied for more than 70 years, it is still unknown whether there is a rule for 
selecting an improving neighbor extreme point, called a \emph{pivot rule}, that guarantees a polynomial 
upper bound on the number of steps performed by the Simplex method. 
In fact, a longstanding open question in the theory of optimization \cite{terlaky+zhang} states:
{\em is there a version of the Simplex method, i.e., a choice of pivot rule, that can solve LPs in strongly-polynomial time}? 

We contribute the following new results about the behavior of the Simplex method. First, we give a partial answer to the above question for
linear programs whose feasible region  $\mcp$ is a 0/1 polytope (in the rest of the paper we will often refer to them as \emph{0/1-LPs}). 
We are able to exploit their \emph{steepest edges}: At a given vertex $\bm{x}$, a steepest edge is given by the edge minimizing $\frac{\bm c^\T \bm g}{||\bm{g}||_1}$ among all the edges $\bm g$ incident at $\bm{x}$. A steepest-descent pivot rule asks to iteratively select a steepest edge as the improving direction.

\begin{thm}
\label{thm:main1} 

Given an LP in the general form~\emph{(\ref{lp})} whose feasible region 
 $\mcp$ is a 0/1 polytope, then
 \begin{enumerate}[(i)] 
\item One can construct a monotone path between any vertex and the optimum via steepest edges, whose length is strongly-polynomial in the input size of the LP. 
\item When the feasible region $\mcp$ is a non-degenerate 0/1 polytope, the
  Simplex method with a steepest-descent pivot rule reaches an optimal solution in strongly-polynomial time. 
\end{enumerate}

\end{thm}

In addition, we consider the following legitimate question: \emph{can one hope to find a pivot rule
that makes the Simplex method use the shortest monotone path?} We here show a negative answer. 

\begin{thm}
\label{thm:main2} Given an LP and an initial feasible solution, finding the shortest (monotone) path to an optimal solution
is NP-hard. Furthermore, unless P=NP, it is hard to approximate within a factor strictly better than two.
\end{thm}

The above theorem implies that for \emph{any} efficiently-computable pivoting rule, an edge-augmentation algorithm 
(like the Simplex method) cannot be guaranteed to reach an optimal solution via a minimum number of augmentations 
(i.e., non-degenerate pivots), unless P=NP.  We remark that these hardness results require the use of degeneracy. 
  The same theorem for non-degenerate polytopes would be a stronger result, and is left as an open question.

\smallskip
The methods we used to prove these two theorems are of independent interest. We arrived to the results above 
while thinking about a much more general family of algorithms that includes the Simplex method.
\emph{Circuit-augmentation} algorithms are extensions of the Simplex method where we have many more choices of 
improving directions available at each step\textemdash more than just the edges of the polyhedron. Our results 
are valid in this more general family of circuit-augmentation algorithms which we now introduce to the reader.

\smallskip
Given a polyhedron, its \emph{circuits} are all potential edges that can arise by translating some of its facets. Circuits are important 
not just in the development of linear optimization \cite{blandthesis,rocka}, but they appear very naturally in other areas of application 
where polyhedra need to be decomposed \cite{muller+regensburger}. Formally:

\begin{defn}
Given a polyhedron of the form $\mcp=\pol$, a non-zero vector $\bm g \in \R^n$ is a circuit if 
\begin{enumerate}[(i)]
	\item $\bm g \in \ker(A)$, and
	\item $B\bm g$ is support-minimal in the collection $\{B \bm y : \bm y \in \ker(A), \bm y \neq 0 \}$.
\end{enumerate}
\end{defn}
Here $\ker(A)$ denotes the kernel of the matrix $A$. As we will see later, from a geometric perspective the circuits can capture the extreme rays of a cone containing any improving direction, and allow one to represent any augmentation as a sum of at most $n$ circuit-augmentations.

To represent the circuits with a finite set, we can normalize them in various ways. 
Following \cite{doi:10.1137/140976868,Borgwardt+Viss,DL15,finhold}, we denote by 
$\mc(A,B)$ the (finite) set of circuits with co-prime integer components. 

\smallskip
Given an initial feasible point of an LP, a circuit-augmentation algorithm at each iteration moves maximally along an improving 
circuit-direction until an optimal solution is found (or unboundedness is detected). Circuits and circuit-augmentation algorithms 
have appeared in several papers and books on linear and integer optimization (see 
\cite{blandflorida,blandedmonds,manycircuitdiameters,doi:10.1137/140976868,Borgwardt+Viss,DHKbook,DL15,GAUTHIERetal2014,Hemmecke+Onn+Weismantel:oracle,doi:10.1137/17M1152115,onn:nonlinear-discrete-monograph,rocka,st96} 
and the many references therein). In particular, the authors of \cite{DL15} considered linear programs in equality form and analyzed in detail three circuit-pivot rules that
guarantee notable bounds on the number of steps performed by a circuit-augmentation algorithm to reach 
an optimal solution.

Given a feasible point $\bm{x} \in \mcp$, the proposed rules are as follows:
\begin{itemize}
\item[(i)] \emph{Greatest-improvement circuit-pivot rule:} select a circuit $\bm{g} \in \mc(A,B)$ 
that maximizes the objective improvement $-\bm c^\T(\alpha\bm g)$, among all circuits $\bm{g}$ and $\alpha \in \mathbb R_{>0}$ such that
$\bm x+\alpha\bm g\in \mcp$.  
\item[(ii)] \emph{Dantzig circuit-pivot rule:} select a circuit $\bm{g} \in \mc(A,B)$ 
that maximizes $-\bm c^\T\bm g$,  among all circuits $\bm{g} $ such that
$\bm x+\varepsilon \bm g\in \mcp$ for some $\varepsilon>0$. 
\item[(iii)] \emph{Steepest-descent circuit-pivot rule:} select a circuit $\bm{g} \in \mc(A,B)$ that 
maximizes $-\frac{\bm c^\T\bm g}{||\bm g||_1}$,  among all circuits $\bm{g}$ such that
$\bm x+\varepsilon \bm g\in \mcp$ for some $\varepsilon>0$.  
\end{itemize}

Note that these circuit-pivot rules are direct extensions to three famous pivot rules proposed 
for the Simplex method. Unfortunately,  
the Simplex method can require an exponential number of edge steps before reaching an optimal solution \cite{GOLDFARB+sit1979,JEROSLOW1973,kleeminty}. 
When all circuits are considered as possible directions to move,  much better bounds can be given. Most notably, 
the greatest-improvement circuit-pivot rule guarantees a \emph{polynomial} bound on the number of steps 
performed by a circuit-augmentation algorithm on LPs in equality form (see  
\cite{DL15,Hemmecke+Onn+Weismantel:oracle,onn:nonlinear-discrete-monograph} and references therein). 
However,  the set of circuits in general can have an exponential cardinality, and therefore selecting the best 
circuit according to the previously mentioned rules is not an easy optimization problem.  
Indeed, the central questions of this paper are the following: 
 
 \begin{itemize}
 \item[$\bullet$]\emph{How hard is it to carry out these three circuit-pivot rules over the exponentially-large set of circuits?}
 \item[$\bullet$] \emph{Can we exploit (approximate) solutions to these circuit-pivot rules to design 
  (strongly-) polynomial time augmentation algorithms?}
 \end{itemize}

  We now describe our results and how they lead to Theorem~\ref{thm:main1} and Theorem~\ref{thm:main2}.
  
\paragraph{Hardness of circuit-augmentation.}  First we settle the computational complexity of the circuit-pivot rules (i) and (ii).

\begin{thm}
\label{thm:hardness}
The greatest-improvement and  Dantzig circuit-pivot rules are NP-hard. 
\end{thm}

We prove this theorem by showing that computing a circuit, according to both the greatest-improvement 
and the Dantzig circuit-pivot rule, is already hard to do when $\mcp$ is a $0/1$ polytope. In particular, we focus on the case
when $\mcp$ is the matching polytope of a bipartite graph. We first characterize the circuits of the more general 
\emph{fractional} matching polytope, i.e., the  polytope given by the standard LP-relaxation for the matching problem 
on general graphs,  in Section~\ref{sec:frac_matching}. This builds on the known graphical characterization of adjacency 
given in \cite{Sanita18,Behrend13}. Then, we construct a reduction from the NP-hard Hamiltonian path problem in 
Section~\ref{sec:hardness}. The heart of the reduction yields the following interesting corollary.

\begin{cor}
\label{cor:bestneighborhard}
Given a feasible extreme point solution of the bipartite matching polytope and an objective function, it is NP-hard to 
decide whether there is a neighbor extreme point that is optimal. 
 \end{cor}

With the above corollary, the hardness results stated in Theorem~\ref{thm:main2} can be easily derived. Even more, 
combining this corollary with the characterization of circuits for the fractional matching polytope, we can show that 
the hardness results of Theorem~\ref{thm:main2} hold more generally for \emph{circuit}-paths, 
i.e., paths constructed by a circuit-augmentation algorithm (see Corollary~\ref{cor:cir}).

\paragraph{Approximating pivoting rules.}
We next make a very useful observation: Any polynomial-time $\gamma$-approximation algorithm for the circuit-pivot 
rule optimization problems yields an increase of at most a $\gamma$-factor on the running time of the corresponding 
circuit-augmentation algorithm -- this follows from an extension of the analysis given by \cite{DL15}. This simple 
observation turns out to be quite powerful in combination with the greatest-improvement circuit-pivot rule, and it plays 
a key role in our subsequent results. We therefore formally state its main implication in the next lemma.

\begin{lem}\label{lem:approximation}
Consider an LP in the general form~\emph{(\ref{lp})}. Denote by $\delta$  the maximum absolute value of 
the determinant of any $n \times n$ submatrix of $\binom{A}{B}$. Let $\bm x_0$ be an initial feasible solution,  
and let $\gamma\geq 1$. Using a $\gamma$-approximate greatest-improvement circuit-pivot rule, we can reach an optimal solution 
$\bm x_{\min}$ of \emph{(\ref{lp})} with $O\big(n\gamma\log\big( \delta \,\bm c^\T (\bm x_0-\bm x_{\min}) \big)\big)$ 
 augmentations.
\end{lem}

We exploit Lemma~\ref{lem:approximation} in two ways. First, as an easy corollary, we get a polynomial bound on the \emph{circuit-diameter}
of any rational polyhedron, where the circuit-diameter is the maximum length of a shortest circuit-path between any two 
vertices. To the best of our knowledge, this has not been observed before. We recall that, as usual, the encoding length of a rational number $\frac{p}{q}$ is defined as  $\lceil \log (p+1)\rceil + \lceil \log (q+1)\rceil + 1$.

\begin{cor}\label{cor:circuit_diameter}
There exists a polynomial function $f(m,\alpha)$ that bounds above the circuit-diameter
of any rational polyhedron $\mcp=\pol$ with $m$ row constraints and maximum encoding length among the coefficients in
its description equal to $\alpha$.
\end{cor}

Second, we prove that for an extreme point $\bm{x}$ of a 0/1-polytope $\mathcal P$, a polynomial approximation of a greatest-improvement circuit-pivot rule can be obtained by simply computing a steepest edge incident at $\bm{x}$.

\begin{thm}
\label{thm:approx_edge} 
Let $\bm{x}$ be an extreme point of a 0/1-LP with $n$ variables. An augmentation along a steepest edge yields an $n$-approximation of an augmentation along a greatest-improvement circuit.
\end{thm}

We stress that the above theorem is not true for arbitrary (non 0/1) LPs, since in general augmenting along the edges incident at a vertex can be arbitrarily worse compared to augmenting along circuits, even in dimension two (see Figure~\ref{fig:tight_approx_for_steep} in Section~\ref{sec:approximating}). 
For 0/1-LPs instead, the above theorem can be combined with Lemma~\ref{lem:approximation} and the seminal result of Frank $\&$ Tardos \cite{andras+eva}, to yield Theorem~\ref{thm:main1}. Even more, we can show that the result of Theorem~\ref{thm:main1} holds more generally for \emph{circuit}-augmentation algorithms (see Theorem~\ref{thm:simplex_0/1}). 

\smallskip
Finally, as a byproduct of our result we also show (i) that a steepest-descent circuit-pivot rule for 0/1-LPs reduces to computing steepest edges, and this can be done in polynomial time (Lemma~\ref{lem:0/1_steep_is_edge}); (ii) a new bound on the number of steps performed by a circuit-augmentation algorithm that uses a steepest-descent circuit-pivot rule, for general (non 0/1) bounded LPs (Theorem~\ref{thm:steep_aprrox_deep}). In fact, the authors of~\cite{DL15},  and later \cite{Borgwardt+Viss}, gave bounds on the number of steps, which depend on the size of the set of circuits $\C(A,B)$ and the number of different values the objective function takes on that set. However, such bounds are a bit opaque, and often difficult to analyze. In Theorem~\ref{thm:steep_aprrox_deep} we get another type of bound of independent application. Our bound is more comparable  to that obtained with the greatest-improvement circuit-pivot rule, since, unlike all previous bounds for steepest-descent, it depends more explicitly on the input description of the LP. Furthermore, it is valid for bounded LPs in general form, and not just LPs in equality form.

\subsection{Relevant prior work and comparison with our results}

We now give some extra context to our results, mentioning several important papers.
\paragraph{Pivoting rules.} Several pivot rules have been proposed for the Simplex method, but since 1972 researchers have debunked most of them as candidates for good polynomial behavior.  As of today, all popular pivoting rules 
are known to require an exponential number of steps to solve some concrete ``twisted'' linear programs (see \cite{amentaziegler,AvisFriedmann2017,Friedmannetal,GOLDFARB+sit1979,JEROSLOW1973,kleeminty,Hansen+Zwick2015,terlaky+zhang,zadeh2,zieglerextremelps}  and the many references there). Some ways to pivot over circuits that we do not consider here are presented e.g., in \cite{McCormick00,SW2002,KarzanovMcCormick97, Borgwardt+Viss}. In particular, the circuit-pivot rule in~\cite{McCormick00} also guarantees geometric convergence. We stress that in this article we consider only ``pure'' pivoting rules that cannot be adjusted along the way (just as the Simplex method requires). If one is allowed to adapt the pivot step along the way, the complexity may be different, but at the same time the algorithm would depart more substantially from the Simplex method. 
Similarly, we restrict ourselves to pivot rules that, while valid for circuits,  
relate to well-known Simplex pivot methods.  Kitahara and Mizuno~\cite{kitaharamizuno} provided a general upper bound for the number of different basic feasible solutions generated by the Simplex method for linear programming problems (LP) having optimal solutions. Their bound is polynomial on the number of constraints, the number of variables, and the ratio between two parameters: $\delta$, the minimum value of a positive component among basic feasible solutions, and $\gamma$, the maximum value among all positive components. Unfortunately, their bound is not polynomial in the encoding length, and in fact, while $\gamma$ can be efficiently computed, it has been shown that $\delta$ is 
NP-hard to determine (see \cite{KunoSanoTsuruda18}).  
Later in \cite{kitaharamatsuimizuno} the authors showed that the bounds are good  for some special LP problems including those on $0/1$ polytopes, those with totally unimodular matrices, and the Markov decision problems. 

\paragraph{0/1 polytopes.} Understanding monotone paths of $0/1$ polytopes is particularly important in discrete optimization.  Thus it is relevant to understand the performance of augmentation algorithms, and in particular of the Simplex method, in that family of polytopes. 
The authors of \cite{Kortenkamp1997} showed that the monotone paths of some pivot rules are unfortunately exponentially long even in $0/1$ polytopes. Later \cite{kitaharamizuno,kitaharamatsuimizuno} proved that the 
number of distinct basic feasible solutions generated by the Simplex method is strongly-polynomial for some 0/1-LPs of special format, namely,
0/1-LPs in standard equality form. For a comparison, our bound in Theorem~\ref{thm:main1} is weaker (though, still strongly-polynomial), but it is valid for \emph{all} 0/1-LPs. 
An important remark is the following: The fact that paths of strongly-polynomial length on the 1-skeleton of $0/1$ polytopes can be 
constructed from \emph{any} augmentation oracle that outputs an improving edge-direction is already known (see \cite{10.1007/3-540-60313-1_164}, and 
\cite{delpiamichini} for a recent generalization to lattice polytopes). However, such arguments require some scaling/modifying of the objective function, and they are not realized with a ``pure'' pivoting rule. In fact, the resulting path might not even be monotone with respect to the original objective function. In contrast, our result in Theorem~\ref{thm:main1} yields a monotone path of strongly-polynomial length that is realized by simply moving along steepest edges. In the context of the Simplex method, our Theorem~\ref{thm:main1} yields a strongly-polynomial bound on the number of distinct basic feasible solutions visited by the Simplex method, if implemented with a pivot rule that makes it move to an adjacent extreme point via a
steepest edge. 
Note that our bound (as well as the results of \cite{kitaharamizuno,kitaharamatsuimizuno}) is not on the total number of \emph{basis exchanges}, in presence of degeneracy. While computing a steepest edge can be done in polynomial time (as explained in the proof of Lemma~\ref{lem:0/1_steep_is_edge}), an interesting open problem is whether moving along steepest edges on $0/1$-LPs 
can be realized efficiently via a ``pure'' pivot rule for the Simplex method in presence of degeneracy. In other words, can one 
guarantee that the number of different adjacent bases visited by the algorithm is strongly-polynomial (rather than the number 
of different extreme points, as we proved here)?

\paragraph{Hardness.} Several authors arrived before us to the conclusion that pivot rules can be used to encode NP-hard problems \cite{adler+papadimitriou+rubinstein,disser+skutella, 
fearnley+savani}. For example, Fearnley and Savani~\cite{fearnley+savani} showed that it is PSPACE-complete to decide whether Dantzig's pivot rule ever chooses a specific  variable to enter the basis. In contrast, Adler et al.~\cite{adler+papadimitriou+rubinstein} showed the simplex method with the shadow vertex pivot rule can  decide this query in polynomial time. 
We stress that our hardness result in Theorem~\ref{thm:main2} does not 
follow from the hardness of computing the diameter of a polytope \cite{FT94,Sanita18}. In fact, 
the hardness results in \cite{FT94} and \cite{Sanita18} rely on the existence/non existence of vertices with a certain 
structure, and they do not provide a specific objective function  to minimize over their polytopes.

 We note that Theorem~\ref{thm:main2} and Corollary \ref{cor:bestneighborhard} can also be 
derived using the circulation polytope~\cite{BT89}, instead of the bipartite matching polytope.
In particular, the characterization of circuits for the circulation polytope becomes easier, 
since circuits can be associated with simple cycles. Hence the hardness results of the above 
oracles can be derived from the NP-hardness of finding a cycle that gives the biggest
improvement in the cost, a result proven in~\cite{BT89}. However, our proof relies on the 
matching polytope for three reasons: First,  bipartite matching is a well-studied combinatorial problem, 
it is quite interesting that the hardness resides already in such a simple and well-studied polyhedron.
Second, if we had written our hardness result using the circulation polytope in \cite{BT89}, the presentation 
would not be any shorter than the proof written for the (integral) bipartite matching polytope. The extra 
material regarding the fractional  matching polytope is only a couple of pages longer, and it does add 
new combinatorial properties to the literature.  Third, we believe that the characterization of the circuits 
of the fractional matching polytope can be useful in determining the complexity of computing the circuit-diameter of a polytope. 
Interestingly, our results in Section~\ref{sec:frac_matching} show that the circuits of the fractional matching polytope 
correspond to the actual edges of the polytope (i.e., the translation of any subset of the facets does not increase 
the initial set of edge-directions). However, we can construct examples showing that the circuit-diameter can be 
strictly smaller than the (combinatorial) diameter for the fractional matching polytope. As a consequence, the 
hardness result on the computation of its diameter given in \cite{Sanita18} does not trivially extend to the circuit setting.
We feel that the characterization of the circuits of the fractional matching polytope can still be exploited to resolve the complexity of the computation of its circuit-diameter, and hence can be useful in attacking the mentioned open question.

\section{Preliminaries} \label{sec:preliminaries}
We consider $\mcp$ to be a polyhedron of the form $\mcp=\pol$ for integer matrices $A$ and $B$ of sizes $m_A \times n$ and $m_B \times n$ respectively, and integer vectors $\bm{b}$ and $\bm{d}$, and assume that we wish to minimize a linear integral objective function $\vecc^\T \bm{x}$ over $\mcp$. We further assume that $A$ has full row rank and that the rank of $\binom{A}{B}$ is $n$. For our purposes, we also assume $m_B \geq 1$ as otherwise $\mcp$ contains trivially only one point. As mentioned before, $\ker(A)$ denotes the kernel of the matrix $A$. 
For a matrix $D$ and a subset $T$ of row indices, we let $D_T$ denote the submatrix of $D$ given by the rows indexed by $T$.
Furthermore, we let $\rk(D)$ denote its rank, and $\det(D)$ denote its determinant. For a vector $\bm{x}$, we let 
 $\supp(\bm{x})$ be the support of the vector $\bm{x}$, and $\bm x(i)$ denote its $i$-th component. 

\smallskip
 Given a vertex $\bar{\bm x}$ of $\mcp$, we define the \emph{feasible cone} at $\bar{\bm x}$ to be the set of all directions $\bm z$ such that $\bar{\bm x}+\varepsilon\bm z\in\mcp$ for some $\varepsilon>0$.  
More formally, it is the set $\set{\bm z \in\R^n: \,A\bm z=\bm 0,\, B_{T(\bar{\bm x})}\bm z\leq \bm 0}$ 
where $T(\bar{\bm x})$ denotes the indices of the inequalities of $B\bm x\leq \bm d$ 
that are tight at $\bar x$. The extreme rays of the feasible cone at $\bar{\bm x}$ are the \emph{edge-directions} at $\bar{\bm x}$.

\smallskip
A circuit-path is a  finite sequence of feasible solutions $\vex_1$, $\vex_2$, $\ldots, \vex_q$   
satisfying $\vex_{i+1}= \vex_{i} + \alpha_{i} \bm{g}_i$, where 
$\bm{g}_i \in \mc(A,B)$ and 
$\alpha_i \in \mathbb R_{>0}$ is such that $\vex_{i} + \alpha_{i} \bm{g}_i \in \mcp$ but $\vex_{i} + (\alpha_{i} + \varepsilon ) \bm{g}_i \notin \mcp$ for all $\varepsilon >0$ (i.e., the augmentation is maximal). 
Note that $\vex_{i}$ is not necessarily a vertex of $\mcp$.
A circuit-path is called monotone if each $\bm g_i$ satisfies $\vecc^\T \bm{g}_i < 0$ (i.e., it is an improving circuit).

\smallskip
A circuit-augmentation algorithm computes a monotone circuit-path starting at a given initial feasible solution,
until an optimal solution is reached (or unboundedness is detected). 
The circuit $\bm{g}$  to use at each augmentation is usually chosen according to some \emph{circuit-pivot rule}.
As discussed before, in this paper we focus on three such rules, each of which gives rise to a corresponding
optimization problem.  

The optimization problem that arises when following the greatest-improvement circuit-pivot rule will be called $Great(\mcp,\bm x, \bm c)$, and is as follows:
\begin{equation*}
	\begin{aligned}
	    &\max - \bm c^\T(\alpha\bm g)\\
	    &\text{s.t.}\\
	    &\bm g\in\mc(A,B),\\
	    &\alpha>0,\\
	    &\bm x+\alpha\bm g\in \mcp . \\
	\end{aligned}
\end{equation*}

The optimization problem that arises when following the Dantzig circuit-pivot rule will be called $Dan(\mcp,\bm x, \bm c)$, 
and is as follows:

\begin{equation*}
	\begin{aligned}
	    &\max - \bm c^\T\bm g\\
	    &\text{s.t.}\\
	    &\bm g\in\mc(A,B),\\
	    &\bm x+\epsilon\bm g\in \mcp && \text{for some }\epsilon>0.
	\end{aligned}
\end{equation*}

The optimization problem that arises when following the steepest-descent circuit-pivot rule will be called $Steep(\mcp,\bm x, \bm c)$, 
and is defined as follows:
\begin{equation*}
	\begin{aligned}
	    &\max - \frac{\bm c^\T\bm g}{||\bm g||_1}\\
	    &\text{s.t.}\\
	    &\bm g\in\mc(A,B),\\
	    &\bm x+\epsilon\bm g\in \mcp && \text{for some }\epsilon>0.
	\end{aligned}
\end{equation*}

A maximal augmentation given by an optimal solution to $Great(\mcp,\bm x, \bm c)$ is called a greatest-improvement augmentation. A  
Dantzig augmentation and a steepest-descent augmentation are defined similarly. In this work, we will only use maximal augmentations, and therefore will omit the word ``maximal''.
We remark that, in the context of pivot rules for the Simplex method, the name `steepest-descent' often refers to 
normalizations according to the 2-norm of a vector, rather than the 1-norm. Here we stick to this name
as used previously in \cite{DL15,Borgwardt+Viss}, but at the end of the paper we note that this is not a major assumption.

\section{Hardness of some Circuit-Pivot Rules}
\label{sec:hardness_pivot_rules}

\subsection{The Circuits of the Fractional Matching Polytope}
\label{sec:frac_matching}

Let $G$ be a simple connected graph with nodes $V(G)$ and edges $E(G)$. We assume $|V(G)|\geq 3$. Given $v\in V(G)$, we let $\delta_G(v)$ denote the edges of $E(G)$ incident with $v$. We call a node $v \in V(G)$ a \emph{leaf} if $|\delta_G(v)|=1$, and let $L(G)$ denote the set of leaf nodes of $G$. Furthermore, for $X\subseteq E$ and $\bm x\in\R^{E(G)}$, we let $\bm x(X)$ denote $\sum_{e\in X}\bm x(e)$.  

Let $\mcp_\fmat(G)$ denote the fractional matching polytope of $G$, which is defined by the following (minimal) linear system:
	\begin{align}
		&\bm{x}\left(\delta_G(v)\right)\leq 1, &&\text{for all } v\in V(G)\setminus L(G). \label{eq1}\\
		&\bm{x} \geq \bm 0.
	\end{align}
In this section, we fully characterize the circuits of $\mcp_\fmat(G)$. 
We will prove that, if $\bm{x}$ is a circuit of 
$\mcp_\fmat(G)$, then $\supp(\bm{x})$ induces a connected subgraph of $G$ that has a very special structure: namely, it belongs to one of the five classes of graphs ($\me_1,\me_2,\me_3,\me_4,\me_5,$) listed below.    
\begin{itemize}
\item[(i)] Let $\me_1$ denote the set of all subgraphs $F\subseteq G$ such that $F$ is an even cycle. 

\item[(ii)]Let $\me_2$ denote the set of all subgraphs $F\subseteq G$ such that $F$ is an odd cycle.

\item[(iii)]Let $\me_3$ denote the set of all subgraphs $F\subseteq G$ such that $F$  is a simple path.

\item[(iv)]Let $\me_4$ denote the set of all subgraphs $F\subseteq G$ such that $F$  is a connected graph satisfying $F=C\cup P$, where $C$ and $P$ are an odd cycle and a non-empty simple path, respectively, that intersect only at an endpoint of $P$. (See Figure~\ref{fig:E4}).
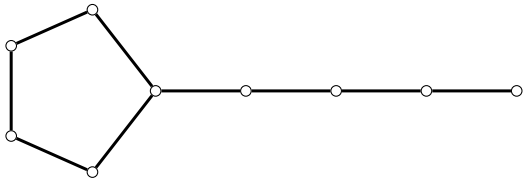
\begin{figure}[!h]
	\begin{center}
			\begin{tikzpicture}[scale=1.2]	
				\tikzset{
					graph node/.style={shape=circle,draw=black,inner sep=0pt, minimum size=4pt
      							  }
						}
				\tikzset{	 edge/.style={black, line width=.4mm
      							  }
						}		

                \node[graph node] (1) at (.2,0){};
                \node[graph node] (2) at (-0.5,0.9){};
                \node[graph node] (3) at (-0.5,-0.9){};
                \node[graph node] (5) at (-1.4,-0.5){};
                \node[graph node] (4) at (-1.4,0.5){};
                
                \node[graph node] (6) at (1.2,0){};
                \node[graph node] (7) at (2.2,0){};
                \node[graph node] (8) at (3.2,0){};
                \node[graph node] (9) at (4.2,0){};
                
                \draw[edge] (1)--(2);
                \draw[edge] (3)--(1);
                \draw[edge] (3)--(5);
                \draw[edge] (5)--(4);
                \draw[edge] (2)--(4);
                
                \draw[edge] (1)--(6);
                \draw[edge] (7)--(6);
                \draw[edge] (7)--(8);
                \draw[edge] (9)--(8);

			\end{tikzpicture}
	\end{center}
	\caption{An Example of a subgraph belonging to $\me_4$}
	\label{fig:E4}
\end{figure}

\item[(v)] Let $\me_5$ denote the set of all subgraphs $F\subseteq G$ such that $F$ is a connected graph with $F=C_1\cup P\cup C_2$, where $C_1$ and $C_2$ are odd cycles, and $P$ is a (possibly empty) simple path satisfying the following: if $P$ is non-empty,
then $C_1$ and $C_2$ are node-disjoint and $P$ intersects each $C_i$ exactly at its endpoints (see Figure~\ref{fig:E5 path});
if $P$ is empty then $C_1$ and $C_2$ intersect only at one node $v$ (see Figure~\ref{fig:E5 no path}).
\begin{figure}[!h]
	\begin{center}
			\begin{tikzpicture}[scale=1.2]	
				\tikzset{
					graph node/.style={shape=circle,draw=black,inner sep=0pt, minimum size=4pt
      							  }
						}
				\tikzset{	 edge/.style={black, line width=.4mm
      							  }
						}		

                \node[graph node] (1) at (.2,0){};
                \node[graph node] (2) at (-0.5,0.9){};
                \node[graph node] (3) at (-0.5,-0.9){};
                \node[graph node] (5) at (-1.4,-0.5){};
                \node[graph node] (4) at (-1.4,0.5){};
                
                \node[graph node] (6) at (1.2,0){};
                \node[graph node] (7) at (2.2,0){};
                \node[graph node] (8) at (3.2,0){};
                \node[graph node] (9) at (4.2,0){};
                
                \node[graph node] (10) at (4.9,.9){};
                \node[graph node] (11) at (4.9,-.9){};
                
                \draw[edge] (1)--(2);
                \draw[edge] (3)--(1);
                \draw[edge] (3)--(5);
                \draw[edge] (5)--(4);
                \draw[edge] (2)--(4);
                
                \draw[edge] (1)--(6);
                \draw[edge] (7)--(6);
                \draw[edge] (7)--(8);
                \draw[edge] (9)--(8);
                
                \draw[edge] (9)--(10);
                \draw[edge] (9)--(11);
                \draw[edge] (11)--(10);

			\end{tikzpicture}

	\caption{An Example of a subgraph belonging to $\me_5$ where $P$ is non-empty.}
	\label{fig:E5 path}

			\begin{tikzpicture}[scale=1.2]	
				\tikzset{
					graph node/.style={shape=circle,draw=black,inner sep=0pt, minimum size=4pt
      							  }
						}
				\tikzset{	 edge/.style={black, line width=.4mm
      							  }
						}		

                \node[graph node] (1) at (.2,0){};
                \node[graph node] (2) at (-0.5,0.9){};
                \node[graph node] (3) at (-0.5,-0.9){};
                \node[graph node] (5) at (-1.4,-0.5){};
                \node[graph node] (4) at (-1.4,0.5){};
                
                \node[graph node] (10) at (0.9,.9){};
                \node[graph node] (11) at (0.9,-.9){};
                
                \draw[edge] (1)--(2);
                \draw[edge] (3)--(1);
                \draw[edge] (3)--(5);
                \draw[edge] (5)--(4);
                \draw[edge] (2)--(4);
                
                \draw[edge] (1)--(10);
                \draw[edge] (1)--(11);
                \draw[edge] (11)--(10);

			\end{tikzpicture}
	\end{center}
	\caption{An Example of a subgraph belonging to $\me_5$ where $P$ is empty.}
	\label{fig:E5 no path}
\end{figure}
\end{itemize}
We will associate a set of circuits to the subgraphs in the above families,
by defining the following five sets of vectors. It is worth noticing that similar elementary moves appeared in \cite{DST1995}
in applications of Gr\"obner bases in combinatorial optimization.

 \begin{equation*}
 \begin{array}{ccccc}
\mc_1= & \bigcup_{F \in \me_1} \Big\{\bm g\in\{-1,1\}^{E(G)}:& \bm g(e) \neq 0 \qquad & \text{iff } e \in E(F) \\ 
&&  \bm g(\delta_F(v))=0 & \forall v \in V(F) & \Big\},\\
&&&&\\
\mc_2=  &\bigcup_{F \in \me_2} \Big\{\bm g\in\{-1,1\}^{E(G)}: & \bm g(e) \neq 0 \qquad & \text{iff } e \in E(F) \\ 
&& \bm g(\delta_F(w))\neq 0 & \text{ for one } w \in V(F) \\
&&  \bm g(\delta_F(v))=0 & \forall v \in V(F)\setminus \{w\} & \Big\},\\
&&&&\\
\mc_3=  &\bigcup_{F \in \me_3} \Big\{\bm g\in\{-1,1\}^{E(G)}: & \bm g(e) \neq 0 \qquad & \text{iff } e \in E(F) \\ 
&&  \bm g(\delta_F(v))=0 & \forall v: |\delta_F(v)| = 2& \Big\},\\
&&&&\\
\mc_4=  &\bigcup_{F=(P\cup C) \in \me_4} \Big\{\bm g\in\mathbb Z^{E(G)}: & \bm g(e) \neq 0 \qquad & \text{iff } e \in E(F) \\ 
&&  \bm g(\delta_F(v))=0 & \forall v: |\delta_F(v)| \geq 2 \\
&& \bm g(e) \in \{-1,1\} & \forall e \in E(C) \\
&& \bm g(e) \in \{-2,2\} & \forall e \in E(P) & \Big\},\\
&&&&\\
\mc_5=  &\bigcup_{F=(C_1\cup P \cup C_2) \in \me_5} \Big\{\bm g\in\mathbb Z^{E(G)}: & \bm g(e) \neq 0 \qquad & \text{iff } e \in E(F) \\ 
&&  \bm g(\delta_F(v))=0 & \forall v \in V(F)  \\
&& \bm g(e) \in \{-1,1\} & \forall e \in E(C_1 \cup C_2) \\
&& \bm g(e) \in \{-2,2\} & \forall e \in E(P) & \Big\}.\\
\end{array}
\end{equation*}

See Figure~\ref{fig:E5 labeled} for an example of a vector $\bm g \in \mc_5$.

\begin{figure}[htb]
	\begin{center}
			\begin{tikzpicture}[scale=.8]	
				\tikzset{
					graph node/.style={shape=circle,draw=black,inner sep=0pt, minimum size=6pt
      							  }
						}
				\tikzset{	 edge/.style={black, line width=.4mm
      							  }
						}		

                \node[graph node] (1) at (0.4,0){};
                \node[graph node] (2) at (-1,1.8){};
                \node[graph node] (3) at (-1,-1.8){};
                \node[graph node] (5) at (-2.8,-1){};
                \node[graph node] (4) at (-2.8,1){};
                
                \node[graph node] (6) at (2.4,0){};
                \node[graph node] (7) at (4.4,0){};
                \node[graph node] (8) at (6.4,0){};
                \node[graph node] (9) at (8.4,0){};
                
                \node[graph node] (10) at (10.4,1.8){};
                \node[graph node] (11) at (10.4,-1.8){};
                
                \draw[edge] (1)--(2) node [midway, fill=white] {$1$};
                \draw[edge] (3)--(1) node [midway, fill=white] {$1$};
                \draw[edge] (3)--(5) node [midway, fill=white] {$-1$};
                \draw[edge] (5)--(4) node [midway, fill=white] {$1$};
                \draw[edge] (2)--(4) node [midway, fill=white] {$-1$}; 
                
                \draw[edge] (1)--(6) node [midway, fill=white] {$-2$};
                \draw[edge] (7)--(6) node [midway, fill=white] {$2$};
                \draw[edge] (7)--(8) node [midway, fill=white] {$-2$};
                \draw[edge] (9)--(8) node [midway, fill=white] {$2$};
                
                \draw[edge] (9)--(10)  node [midway, fill=white] {$-1$};
                \draw[edge] (9)--(11)  node [midway, fill=white] {$-1$};
                \draw[edge] (11)--(10)  node [midway, fill=white] {$1$};

			\end{tikzpicture}
	\end{center}
	\caption{Example of a vector $\bm g \in \mc_5$.  Each edge $e$ is labeled with $\bm g(e)$.}
	\label{fig:E5 labeled}
\end{figure}
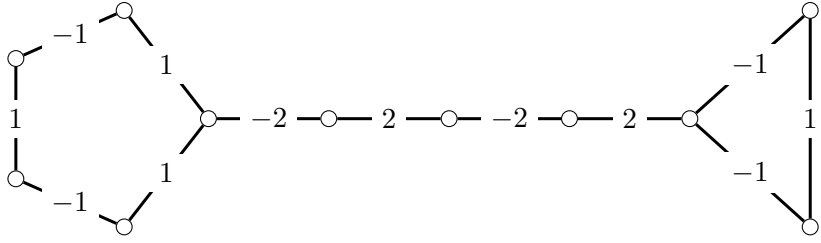

Let us denote by $\mc(\mcp_\fmat(G))$ the set of circuits of $\mcp_\fmat(G)$ with co-prime integer components. 
\begin{lem}
$\mc(\mcp_\fmat(G))=\mc_1\cup\mc_2\cup\mc_3\cup\mc_4\cup\mc_5$.
\end{lem}

\begin{proof}
It is known that the vectors of $\mc_1\cup\cdots\cup\mc_5$ correspond to edge-directions of $\mcp_\fmat(G)$ (see e.g.~\cite{Behrend13,Sanita18}), so it remains to be shown that all circuits belong to one of these sets.  

Let $\B$ denote the constraint matrix corresponding to the inequality constraints~(\ref{eq1}). 
In what follows, the rows of $\B$ will be indexed by $V(G)\setminus L(G)$, and the columns of $\B$ will be indexed by $E(G)$. 
With this notation, we can treat $\supp(\B\bm x)$ and $\supp(\bm x)$ as a subset of $V(G)$ or 
$E(G)$, respectively. Let $\bm g\in\mc(\mcp_\fmat(G))$, and let $G(\bm g)$ be the subgraph of $G$ induced by the edges in $\supp(\bm g)$.

\smallskip
First we note that $G(\bm g)$ is connected.  Otherwise, restricting $\bm g$ to the edges of any component of $G(\bm g)$ gives a vector $\bm f$ with $\supp(\B\bm f)\subseteq\supp(\B\bm g)$ and $\supp(\bm f)\subsetneq\supp(\bm g)$, contradicting that $\bm g$ is a circuit.

\smallskip
Now, suppose that $G(\bm g)$ contains no cycles. Let $P$ be any edge-maximal path in $G(\bm g)$, with endpoints $u$ and $w$.  
Note that $\supp(\B\bm g) \supseteq \{u,w\}\setminus L(G)$.
Let $\bm f \in \{-1,1\}^{E(G)}$ be a vector that satisfies (i) $\bm f(e) \neq 0$ if and only if $e \in E(P)$, and (ii) $\bm f(\delta_P( v))=0 \;
 \forall  v\neq u,w$. Note that $\bm f\in \mc_3$. Then, $\supp(\B\bm f) = \{u,w\} \setminus L(G) \subseteq \supp( \B \bm g)$, and $\supp(\bm f) \subseteq\supp(\bm g)$. 
 Therefore, it must be that the edges of $G(\bm g)$ are exactly $E(P)$, and $\bm g(\delta_P(v))=0$ for all $v\in V(G)\setminus\{u,w\}$.  
Thus, $\bm g = \bm f$ or $\bm g = - \bm f$. In any case, $\bm g\in \mc_3$.

\smallskip
Now, suppose that $G(\bm g)$ contains an even cycle $C$. Let $\bm f \in \{-1,1\}^{E(G)}$ be a vector that satisfies (i) $\bm f(e) \neq 0$ if and only if $e \in E(C)$, and (ii) $\bm f(\delta_C( v))=0 \;
 \forall  v\in V(C)$. Note that $\bm f\in \mc_1$. Then, $\supp(\B\bm f)=\emptyset \subseteq \supp( \B \bm g)$, and $\supp(\bm f) \subseteq\supp(\bm g)$. 
 Therefore, it must be that the edges of $G(\bm g)$ are exactly $E(C)$, and $\bm g(\delta_C(v))=0$ for all $v\in V(G)$.  
Thus, $\bm g = \bm f$ or $\bm g = - \bm f$. In any case, $\bm g\in \mc_1$.

\smallskip
We are left with the case where $G(\bm g)$ contains at least one cycle, 
but it does not contain any even cycle. In this case, first we state an easy claim 
that gives some more structure for the graph $G(\bm g)$. 

\begin{claim} Under the assumption that $G(\bm g)$ contains at least one cycle, but it does not contain an even cycle, 
any two odd cycles in $G(\bm g)$ must share at most one node. 
\end{claim}

\begin{proof}
Let $C,D\subseteq G(\bm g)$ be two odd cycles, and suppose for the sake of contradiction that $|V(C)\cap V(D)|\geq 2$.  Then $C$ can be written as the union of two edge-disjoint paths $C_1\cup C_2$ where $C_1$ is some sub-path of $C$ such that $V(C_1)\cap V(D)=\{u,v\}$ where $u$ and $v$ are the endpoints of $C_1$, and $E(C_1) \cap E(D) =\emptyset$.  Since $D$ is a cycle, we can decompose $D$ into two sub-paths $D_1$ and $D_2$ each with endpoints $u$ and $v$.  Since $|E(D)|$ is odd, for \textit{exactly} one $i\in\{1,2\}$, $|E(D_i)|$ is even.   Note that since $V(C_1)\cap V(D_i)=\{u,v\}$, $C_1\cup D_i$ is a cycle for all $i\in\{1,2\}$, and therefore there exists $i\in\{1,2\}$ such that $C_1\cup D_i$ is an even cycle, a contradiction with the assumption.
\end{proof}

Suppose that $G(\bm g)$ contains at least two distinct odd cycles $C_1$ and $C_2$.  Since $G(\bm g)$ is connected, then either these two cycles share a node or there exists a simple path $P$ in $G(\bm g)$ connecting them.  In particular, we can choose $P$ so that $E(P)\cap E(C_i)=\emptyset$ for $i\in\{1,2\}$.  Let $F=C_1\cup P\cup C_2$ (where $E(P)=\emptyset$ if $C_1$ and $C_2$ share a node). Let $\bm f \in \mathbb Z^{E(G)}$ be a vector that satisfies (i) $\bm f(e) \neq 0$ if and only if $e \in E(F)$,  (ii) $\bm f(\delta_F( v))=0 \; \forall  v\in V(F)$, 
(iii) $\bm f(e) \in \{-1,1\}$ for all $e \in E(C_1 \cup C_2)$, and (iv) $\bm f(e) \in \{-2,2\}$ for all $e \in E(P)$.  Note that $\bm f\in \mc_5$.
Then $\supp(\B\bm f)=\emptyset \subseteq \supp(\B \bm g)$, and $\supp(\bm f)\subseteq\supp(\bm g)$. Therefore, it must be that the edges of $G(\bm g)$ are exactly $E(F)$, and $\bm g(\delta_G(v))=0$ for all $v\in V(G)$. Thus, $\bm g\in \mc_5$.

\smallskip
Finally, suppose that $G(\bm g)$ contains exactly one odd cycle $C$.  If there exists a node $w\in V(C)$ such that $\bm g(\delta_G(w))\neq 0$, then let  $\bm f \in \{-1,1\}^{E(G)}$ be a vector that satisfies (i) $\bm f(e) \neq 0$ if and only if $e \in E(C)$, and (ii) $\bm f(\delta_C( v))=0 \;
 \forall  v\in V(C)\setminus \{w\}$. Note that $\bm f\in \mc_2$.
  Then, $\supp(\B\bm f)=\{w\} \subseteq \supp(\B\bm g)$, and $\supp(\bm f)\subseteq\supp(\bm g)$. Therefore, it must be that the edges of $G(\bm g)$ are exactly $E(C)$, and $\bm g(\delta_C(v))=0$ for all $v\in V(G)\setminus\{w\}$.  Thus, it must be that $ \bm g\in \mc_2$.

\smallskip
We are left with the case where $\bm g(\delta_G(v))=0$ for all $v\in V(C)$.  Note that this is not possible if $\supp(\bm g)=E(C)$, because $C$ is an odd cycle.  Then let $P$ be any simple path in $G(\bm g)$ which is inclusion-wise maximal subject to the condition that $E(P)\cap E(C)=\emptyset$ and 
$|V(P)\cap V(C)|=\{u\}$, where $u$ is an endpoint of $P$. Let $F=C\cup P$, and let $w\in V(G)$ be the unique node such that $|\delta_F(w)|=1$. 
Let $\bm f \in \mathbb Z^{E(G)}$ be a vector that satisfies (i) $\bm f(e) \neq 0$ if and only if  $e \in E(F)$,  (ii) $\bm f(\delta_F( v))=0 \; \forall  v\in V(F)\setminus\{w\}$, (iii) $\bm f(e) \in \{-1,1\}$ for all $e \in E(C)$, and (iv) $\bm f(e) \in \{-2,2\}$ for all $e \in E(P)$.  Note that $\bm f\in \mc_4$.
  Then $\supp(\B\bm f) = \{w\} \setminus L(G) \subseteq \supp(\B\bm g) $, and $\supp(\bm f)\subseteq\supp(\bm g)$. Therefore, it must be that the edges of $G(\bm g)$ are exactly $E(F)$, and $\bm g(\delta_F(v))=0$ for all $v\in V(G)\setminus\{w\}$.  Thus, it must be that $\bm g\in \mc_4$.

\smallskip
In all the above cases, $\bm g\in\mc_1\cup\cdots\cup\mc_5$, as desired.
\end{proof}

\subsection{Hardness Reduction}
\label{sec:hardness}

We will start by proving hardness for the Dantzig circuit-pivot rule.
\begin{thm}\label{thm:oracle1_np_hard}
Solving the optimization problem $Dan(\mcp,\bm x, \bm c)$ is NP-hard. 
\end{thm}
\begin{proof}
We will prove this via reduction from the directed Hamiltonian path problem.  Let $D=(N,F)$ be a directed graph with 
$n=|N|$, and let $s,t \in N$ be two given nodes.  We will construct a suitable auxiliary undirected graph $H$, cost function $\bm c$, and a matching $M$ in $H$, such that the following holds: $D$ contains a directed Hamiltonian $s,t$-path if and only if an optimal solution to $Dan(\mcp_\fmat(H),\chi^M,\bm c)$ (where $\chi^M$ is the characteristic vector of $M$) attains a certain objective function value.

\smallskip
We start by constructing $H=(V,E)$.  For each node $v\in N\setminus\{t\}$ we create two copies $v_a$ and $v_b$ in $V$. For all $v\in N\setminus\{t\}$, we let $v_av_b\in E$.  For all arcs $uv\in F$, with $u,v\neq t$, we add an edge $u_bv_a\in E$. That is, every in-arc at a node $v$ corresponds to an edge incident with $v_a$, and every out-arc at $v$ corresponds to an edge incident with $v_b$.  We add $t$ in $V$, and for all arcs $ut\in F$, we have that $u_bt\in E$.  Finally, we add nodes $s'$ and $t'$, where $s's_a\in E$ and $tt'\in E$ (see Figure~\ref{fig:one one one case C}).

 \begin{figure}[!ht]
	\begin{center}
			\begin{tikzpicture}[scale=1.2]	
				\tikzset{
					graph node/.style={shape=circle,draw=black,inner sep=0pt, minimum size=4pt
      							  }
						}
				\tikzset{	plus edge/.style={red, , line width=.6mm
      							  }
						}		
				\tikzset{	minus edge/.style={blue, line width=.8mm, dashed
	      						  }
						}	
				\tikzset{	graph edge/.style={black, line width=.6mm
	      						  }
						}
						
				\node[graph node, label=$s'$] (s') at (-1,.5){};
				\node[graph node, label=$s_a$] (s1) at (0,0){};
				\node[graph node, label={below:$s_b$}] (s2) at (0,-1){};
				\node[graph node, label=$v_a$] (v1) at (1,0){};
				\node[graph node, label={below:$v_b$}] (v2) at (1,-1){};
				\node[graph node] (u1) at (2,0){};
				\node[graph node] (u2) at (2,-1){};
				\node[graph node] (w1) at (3,0){};
				\node[graph node] (w2) at (3,-1){};
				\node[graph node] (x1) at (4,0){};
				\node[graph node] (x2) at (4,-1){};
				\node[graph node] (y1) at (5,0){};
				\node[graph node] (y2) at (5,-1){};
				\node[graph node, label=$t$] (t) at (6,0){};
				\node[graph node, label=$t'$] (t') at (7,.5){};
                \draw[minus edge] (s')--(s1);
                \draw[minus edge] (t')--(t);
                \draw[graph edge] (s2)--(v1);
                \draw[graph edge] (v2)--(u1);
                \draw[graph edge] (u2)--(v1);
                \draw[graph edge] (u2)--(w1);
                \draw[graph edge] (w2)--(x1);
                \draw[graph edge] (w1)--(x2);
                \draw[graph edge] (x2)--(y1);
                \draw[graph edge] (y2)--(t);
                
                \draw[plus edge] (s1)--(s2);
                \draw[plus edge] (v1)--(v2);
                \draw[plus edge] (u1)--(u2);
                \draw[plus edge] (w1)--(w2);
                \draw[plus edge] (x1)--(x2);
                \draw[plus edge] (y1)--(y2);
			\end{tikzpicture}
	\end{center}
	\caption{An example of the auxiliary graph $H$.}
	\label{fig:one one one case C}
\end{figure}

Now we define the cost function $\bm c$. We set $\bm c(v_av_b)=0$ for all $v\in N\setminus\{t\}$, $\bm c(s's_a)=-W=-\bm c(tt')$ such that $W\in\Z$, $W\gg|E|$, and let all other edges have cost $-1$. Finally, we let 
$$
M=\big\{v_av_b: v\in N\setminus\{t\}\big\}\cup\{tt'\}
$$
be a matching in $H$. We claim that there exists a directed Hamiltonian $s,t$-path  in $D$ if and only if there is a solution
$\bm g$ to $Dan(\mcp_\fmat(H),\chi^M,\bm c)$ with objective function value at least $-\bm c^\T \bm g = 2W + n-1$.

\smallskip
($\Rightarrow$) Suppose that there exists a directed Hamiltonian $s,t$-path $P=(sv^1,v^1v^2,\cdots,v^{k-1}v^k,v^kt)$ in $D$.  
Then, $P$ can be naturally associated to an $M$-alternating path $P'$ in $H$ with endpoints $s'$ and $t'$, as follows: 
$$
P'=(s's_a,s_as_b,s_bv^1_a,v^1_av^1_b,v^1_bv^2_a,v^2_av^2_b,\cdots,v^{k-1}_av^{k-1}_b,v^{k-1}_bv^k_a,v^k_av^k_b,v^k_bt,tt').
$$  
Let $\bm g$ be defined as
$$
	\bm g(e):=\begin{cases}
				1 &\text{  if  } e\in  E(P')\setminus M,\\ 
				-1 &\text{ if } e\in M, \\ 
				0 &\text{otherwise}.
			\end{cases}\,
$$
Then $\bm g \in \mc_3$, and is therefore a circuit of $\mcp_\fmat (H)$. 
Note that $\chi^M + \bm g \in \mcp_\fmat (H)$, and $-\bm c^\T \bm g = 2W + n-1$. Thus, 
$\bm g$ is a feasible solution to $Dan(\mcp_\fmat(H),\chi^M,\bm c)$ with the claimed objective function
value.

\smallskip
($\Leftarrow$) Now suppose that there is a solution $\bm g$ to $Dan(\mcp_\fmat(H),\chi^M,\bm c)$, with objective function
value at least $2W + n-1$. First, we argue that the support of $\bm g$ is indeed an $M$-alternating path
with endpoints $s'$ and $t'$.

Note that, by construction, $H$ is bipartite, so $\bm g \in \mc_1 \cup \mc_3$. 
In either case, $\bm g\in\{1,0,-1\}^E$.  By our choice of $W$, since $-\bm c^\T\bm g \geq 2W + n-1$, it must be that $\bm g(s's_a)=1$ and $\bm g(tt')=-1$.  Then, since $s'$ and $t'$ are not in any cycles of $H$, necessarily $\bm g \in \mc_3$ and its support is an $s',t'$-path.  It follows that $\bm g$ has at most $|V|-1$ non-zero entries.  Two of the non-zero entries are $\bm g(s's_a)$ and $\bm g(tt')$, and of those that remain, exactly half have value 1. Thus, 
$$
\bm -c^\T \bm g\leq 2W+\frac{1}{2}(|V|-3)=2W+\frac{1}{2}((2n+1)-3)=2W+n-1.
$$
It is clear that the above inequality holds tight only if $\bm g(e)=1$ for $\frac{1}{2}(|V|-3)$ edges of $E\setminus\{s's_a,tt'\}$, all of which have $\bm c(e)=-1$, and $\bm c(f)=0$ for all edges $f$ such that $\bm g(f)=-1$. Since the number of edges $e$ with $\bm g(e)=1$ equals the number of edges $f$ with $\bm g(f)=-1$, we have that $|\supp(\bm g)|=|V|-1$, and therefore $\supp(\bm g)$ is a path $P'$ spanning $H$. Furthermore,
all edges of $M$ are in $E(P')$. By removing the first and the last edge
of $P'$, and by contracting all edges of $M$ that are the form $(v_a v_b)$ (for $v \in N$), we obtain a path that naturally corresponds to a directed  Hamiltonian $s',t'$-path in $D$. 
\end{proof}

Note that the above proof immediately yields the following theorem as a corollary.
\begin{thm}
\label{thm:oracle2_np_hard}
Solving the optimization problem $Great(\mcp,\bm x, \bm c)$ is NP-hard.
\end{thm}
\begin{proof}
The proof is identical to that of Theorem~\ref{thm:oracle1_np_hard}, we only need to replace 
$Dan(\mcp_\fmat(H),\chi^M,\bm c)$ with $Great(\mcp_\fmat(H),\chi^M,\bm c)$.  This is because for any
circuit $\bm y \in \mc (\mcp_\fmat(H))$, we have $\chi^M + 1 \bm y \in \mcp$, and $\chi^M +\alpha \bm y \notin \mcp$ for any $\alpha > 1$.
Therefore, for all $\bm y \in \mc (\mcp_\fmat(H))$ such that $-\bm c^\T \bm y >0$, we have
$$
\max \{ -\bm c^\T ( \alpha \bm y) : \chi^M + \alpha \bm y \in \mcp_\fmat(H) , \, \alpha>0 \} = -\bm c^\T\bm y.
$$
It is not difficult to see that this implies the result.
\end{proof}
We highlight that these hardness results hold indeed for $0/1$ polytopes. In fact, since by our construction
the graph $H$ is bipartite, the polytope $\mcp_\fmat(H)$ is integral.

\subsection{Hardness implications}
\label{sec:alg_cons}
Here we prove that the reductions in the previous section have interesting hardness implications for the Simplex method.

\begin{cor}
\label{cor:}
Given a feasible extreme point solution of a $0/1$-LP, computing the best neighbor extreme point is NP-hard.
\end{cor}
\begin{proof}
Consider again the hardness reduction used in the proof of Theorem~\ref{thm:oracle2_np_hard},
and note that the optimal solution of $Great(\mcp_\fmat(H),\chi^M,\bm c)$ is a circuit $\bm g$ that corresponds to an edge-direction at $\chi^M$. As a consequence, if we consider the LP obtained by minimizing $\bm{c^\T}\bm{x}$ over $\mcp_\fmat(H)$, and take $\chi^M$ as an initial vertex solution,
there is a neighbor optimal solution of objective function value $-W -n  +1$ (which is the minimum possible value) if and only if the initial directed graph has a Hamiltonian path. The result follows.
\end{proof}

Note that the above proof also yields a proof for Corollary~\ref{cor:bestneighborhard}. We can now prove Theorem~\ref{thm:main2}, 
that we restate for convenience.

\newtheorem*{thm02}{Theorem~\ref{thm:main2}}
\begin{thm02} Given an LP and an initial feasible solution, finding the shortest (monotone) path to an optimal solution is NP-hard. Furthermore, unless P=NP, it is hard to approximate within a factor strictly better than two.
\end{thm02}

\begin{proof}
Once again, consider the hardness reduction used in the proof of Theorem~\ref{thm:oracle2_np_hard}, and the LP obtained by minimizing $\bm{c^\T}\bm{x}$ over $\mcp_\fmat(H)$. In order for a Hamiltonian path to exist on $D$, the optimal solution of this LP must have objective function
value $-W -n  +1$, so without loss of generality, we can assume that this is the case. Take $\chi^M$ as the initial vertex solution. Under the latter assumption, as noted in the proof of the previous corollary, there is a neighbor optimal solution to $\chi^M$ if and only if $D$ has a Hamiltonian path. This implies the following: (i) if $D$ has a Hamiltonian path, then there is a shortest (monotone) path to an optimal solution on the 1-skeleton of $\mcp_\fmat(H)$, that consists of one edge; (ii) if $D$ does not have a Hamiltonian path, then any shortest (monotone) path to an optimal solution has at least two edges. The result follows.
\end{proof}

As mentioned in the introduction, our result implies that for any efficiently-computable pivoting rule, the Simplex 
method cannot be guaranteed to reach an optimal solution via a minimum number of non-degenerate pivots, unless P=NP. In a way, 
this result is similar in spirit to some hardness results proven about the vertices that the Simplex method can visit during its execution \cite{fearnley+savani,disser+skutella,adler+papadimitriou+rubinstein}.

\smallskip
The latter hardness result also holds for circuit-paths, via the exact same argument.
\begin{cor}
\label{cor:cir}
Given an LP and an initial feasible solution, finding the shortest (monotone) circuit-path to an optimal solution is NP-hard. 
Furthermore, unless P=NP, it is hard to approximate within a factor strictly better than two.
\end{cor}

\smallskip
Note that Theorem~\ref{thm:oracle1_np_hard} and Theorem~\ref{thm:oracle2_np_hard} yield a proof of Theorem~\ref{thm:hardness}.

\section{Approximation of Circuit-Pivot Rules}\label{sec:approximating}

We start with the following formal definition of approximate greatest-improvement augmentations.

\begin{defn}
Let $\gamma \geq 1$, $\bm x \in \mcp$, and $\alpha^*\bm g^*$ be a greatest-improvement augmentation at $\bm x$. We say that an augmentation $\alpha\bm g$ is a $\gamma$-approximate greatest-improvement augmentation at $\bm x$, if 
$$
c^{\T}\bm x-c^{\T}(\bm x +\alpha\bm g)\geq\frac{1}{\gamma}\Big(c^{\T}\bm x-c^{\T}(\bm x +\alpha^*\bm g^*)\Big).
$$
\end{defn}

As mentioned in the introduction, we define
$$
\delta\coloneqq\max\left\{\left|\det\binom{A}{D}\right|\right\},
$$ 
where the max is taken over all $n\times n$ submatrix $\binom{A}{D}$ of $\binom{A}{B}$ such that $\binom{A}{D}$ has rank $n$. 

\smallskip
We also define
$$\bar \delta\coloneqq\max\left\{\left|\det\binom{\bar A}{\bar D}\right|\right\},
$$ 
where the max is taken over all $n\times n$ submatrix $\binom{\bar A}{\bar D}$ of $\binom{A | \bm{b}}{B | \bm{d}}$ such that $\binom{\bar A}{\bar D}$ has rank $n$. 

\smallskip
Furthermore, we let $\omega_1$ be the minimum 1-norm distance from any extreme point to any facet not containing it.
Formally, let $vert(\mcp)$ be the set of vertices of $\mcp$. For a given $\bm v \in vert(\mcp)$, let $\mathcal{F}(\bm v)$ be the set of feasible points of $\mcp$ that lie on any facet $F$ of $\mcp$ with $\bm v \notin F$.

$$
\omega_1:= \min_{\mbox{$\bm v \in vert(\mcp)$, $\bm f \in \mathcal{F}(\bm v)$ } } \| \bm v - \bm f\|_1.
$$

Finally, we let  $M_1$ be the maximum 1-norm distance between any pair of extreme points, i.e.
$$
M_1:= \max_{\mbox{ \small $\bm v_1$,$\bm v_2 \in vert(\mcp)$ } } \| \bm v_1 - \bm v_2\|_1.
$$

\subsection{Approximate greatest-improvement augmentations}
\label{sec:apx}
Let us recall the statement of Lemma~\ref{lem:approximation}: 

\newtheorem*{lemapprox}{Lemma~\ref{lem:approximation}}

\begin{lemapprox}
Consider an LP in the general form~\emph{(\ref{lp})}. Denote by $\delta$ the maximum absolute value of the determinant 
of any $n \times n$ submatrix of $\binom{A}{B}$.  Let $\bm x_0$ be an initial feasible solution,  and let $\gamma\geq 1$.
Using a $\gamma$-approximate greatest-improvement circuit-pivot rule, we can reach an optimal solution $\bm x_{\min}$ 
of \emph{(\ref{lp})} with with $O\big(n\gamma\log\big( \delta \,\bm c^\T (\bm x_0-\bm x_{\min}) \big)\big)$ augmentations.
\end{lemapprox}

The proof of Lemma~\ref{lem:approximation} closely mimics the arguments used in \cite{DL15}. However, since the authors of~\cite{DL15} consider LPs in equality form, for the sake of completeness we will re-state (and sometimes reprove) some of their lemmas for our more general setting. Indeed, when working with circuits, converting an LP to equality form by adding slack variables cannot be done without loss of generality, since this operation might increase the number of circuits (see \cite{Borgwardt+Viss}).

\smallskip
The first proposition that we state is the \emph{sign-compatible representation property} of circuits.  We say two vectors $\bm v$ and $\bm w$ are \emph{sign-compatible with respect to} $B$ if the $i$-th components of the vectors $(B\bm v)$ and $(B\bm w)$ satisfy $(B\bm v)(i)\cdot(B\bm w)(i)\geq 0$ for all $1\leq i\leq m_B$.  The representation property is as follows:

\begin{prop}[see Proposition 1.4 in~\cite{finhold}]\label{prop:sign_compat_sum}
Let $\bm v\in\ker(A)\setminus\set{\bm 0}$. Then we can express $\bm v$ as $\bm v=\sum_{i=1}^k \alpha_i\bm g^i$ such that for all $1\leq i \leq k$
\begin{itemize}
    \item $\bm g^i\in\C(A,B)$,
    \item $\bm g^i$ and $\bm v$ are sign-compatible with respect to $B$ and $\supp(B\bm g^i)\subseteq\supp(B\bm v)$,
    \item $\alpha_i\in\R_{\geq 0}$,
    \item and $k\leq n$.
\end{itemize}
\end{prop}

Let $\bm x_{\max}$ be a maximizer
of the LP problem (\ref{lp}), i.e., an optimal solution of the LP obtained from (\ref{lp}) by multiplying
the objective function by $-1$. We will use the following lemma from~\cite{DL15} based on well-known estimates 
of \cite{Ahuja+Magnanti+Orlin}:

\begin{lem}[see Lemma 1 in~\cite{DL15}]
\label{lem:within_epsilon}
Let $\epsilon>0$ be given. Let $\bm c$ be an integer vector.
 Define $f^{\min}:=\bm c ^\T \bm x_{\min}$, $f^{\max}:=\bm c ^\T \bm x_{\max}$.
Suppose that $f^k=\bm c ^\T \bm x_k$ is the objective function value of the solution $\bm x_k$ at the $k$-th iteration of an
augmentation algorithm. Furthermore, suppose that the algorithm guarantees that for every augmentation $k$,
\[
(f^k-f^{k+1})\geq\beta(f^k-f^{\min}).
\] 
Then the algorithm reaches a solution with $f^k-f^{\min}<\epsilon$ in no more than $2 \log{((f^{\max}-f^{\min})/\epsilon)}/\beta$ augmentations.
\end{lem}

We now state the following easy lemma, that we reprove for completeness. 

\begin{lem}\label{cor:snap_to_vert} 
Let $\bar{\bm x}$ be any feasible solution of the LP problem~\emph{(\ref{lp})}.  Then with a sequence of at most $n$ maximal augmentations, we can reach an extreme point $\hat{\bm x}$ of~\emph{(\ref{lp})} such that $\bm c^\T\hat{\bm x}\leq\bm c^\T\bar{\bm x}$.
\end{lem}

\begin{proof}
Let $T=\set{i:B_i\bar{\bm x}=\bm d(i)}$.  
If $\bar{\bm x}$ is not a vertex, then we can select any direction $\bm g\in \ker\binom{A}{B_T}$ such that $\bm c^\T\bm g\leq0$, 
and such that for some $\epsilon>0$, $\tilde{\bm x}:=  \bar{\bm x}+\epsilon\bm g$ satisfies $B_i\tilde{\bm x}\leq \bm d(i)$ for all $i \notin T$.
We then use $\bm g$ to perform a maximal step $\alpha\bm g$ at $\bar{\bm x}$.  Since the step is maximal, there exists an index $i\notin T$ such that $B_i(\bar{\bm x}+\alpha\bm g) =\bm d(i)$.   This enables us to grow the set $T$ at the new feasible solution. Furthermore, $\bm c^\T(\bar{\bm x}+\alpha\bm g)\leq\bm c^\T\bar{\bm x}$.\\
We can iterate this process, and note that the number of linearly independent rows of $\binom{A}{B_T}$ increases by one at each step. Therefore, after at most $n-\rk(A)$ iterations we arrive at a vertex $\hat{\bm x}$.

Note that the above argument does not require the use of circuits, but it requires only that the selected directions are improving with respect to $\bm c$.  By the sign-compatible representation property of circuits though, at any non-optimal point $\bar{\bm x}$, there always exists an improving direction that is a circuit.  
\end{proof}

\smallskip

We can now give a proof of Lemma~\ref{lem:approximation}.

\begin{proof}[Proof of Lemma~\ref{lem:approximation}]
Let $\bm{x}_k$ be the solution at the $k$-th iteration of an augmentation algorithm.  By the sign-compatible representation property of the circuits,

$$\bm x_{\min}-\bm x_k = \sum_{i=1}^p \alpha_i\bm g^i $$
where $\bm g^i \in \C(A,B)$ and  $p \leq n$. Note that, as a consequence of Proposition \ref{prop:sign_compat_sum}, 
for any $i$ we have that $\alpha_i \geq 0 $ and $\bm x_k+\alpha_i \bm g^i$ is feasible. More precisely, 
$A(\bm x_k+\alpha_i \bm g^i)=A \bm x_k+\alpha_i A \bm g^i=A \bm x_k= \bm b$. Furthermore, we know $B \bm x_k \leq \bm d$
and $B (\bm x_{\min})= B(\bm x_k+ \sum_{i=1}^p \alpha_i \bm g^i) \leq \bm{d}$. This and the sign-compatibility of $\bm g^i$ and 
$(\bm x_{\min}-\bm x_k)$ implies that $B(\bm x_k+\alpha_i \bm g^i) \leq \bm d$ for all $i=1 \dots p$.

We then have
\[
  0>\bm c^\T (\bm x_{\min}-\bm x_k)
   =\bm c^\T \sum_{i=1}^p\alpha_i\bm g^i
   = \sum_{i=1}^p \alpha_i\bm c^\T\bm g^i
   \geq -n\Delta,
\]
where $\Delta>0$ is the largest value of $-\alpha\bm c^\T\bm z$ over all $\bm z\in\C(A,B)$  and $\alpha>0$ for which 
$\bm x_k+\alpha\bm z$ is feasible. Those circuit vectors and coefficients include the $\bm g^i, \alpha_i$ above in particular. 
Equivalently, we get
\[
  \Delta\geq\frac{\bm c^\T (\bm x_k-\bm x_{\min})}{n}.
\]

Now let $\alpha\bm z$ be a $\gamma$-approximate greatest-improvement augmentation applied to $\bm x_k$, leading to $\bm x_{k+1}:=\bm x_k+\alpha\bm z$. 
Since $-\alpha\bm c^\T\bm z\geq \frac{1}{\gamma} \Delta$, we get
\[
  \bm c^\T (\bm x_k-\bm x_{k+1})=-\alpha\bm c^\T\bm z\geq\frac{1}{\gamma}\Delta
  \geq\frac{\bm c^\T (\bm x_k-\bm x_{\min})}{\gamma n}.
\]
Thus, we have at least a factor of $\beta=\frac{1}{\gamma n}$ of objective function value decrease at each augmentation. Applying~\autoref{lem:within_epsilon} with $\epsilon=1/\delta^2$ then yields a solution $\bar{\bm x}$ with $\bm c^\T(\bar{\bm x}-\bm x_{\min})<1/\delta^2 $, obtained within at most $4n\gamma\,\log(\delta\,\bm c^\T (\bm x_0-\bm x_{\min}))$ augmentations.

\smallskip
By \autoref{cor:snap_to_vert}, a vertex
solution $\bm x'$ with $\bm c^\T\bm x'\leq\bm c^\T\bar{\bm x}$ can be reached from $\bar{\bm x}$ in at most $n$ additional augmentations.  It remains to prove that  $\bm x'$ is optimal.

\smallskip
Suppose $\bm x'$ is a non-optimal vertex.  There exist subsets $T_1$ and $T_2$ of $\set{1,\ldots,m_B}$ such that $\bm x'$ is the unique solution to $$\binom{A}{B_{T_1}}\bm x = \binom{\bm b}{\bm d(T_1)},$$ and $\bm x_{\min}$ is the unique solution to $$\binom{A}{B_{T_2}}\bm x=\binom{\bm b}{\bm d(T_2)}.$$ 
Let $\delta_1=|\det\binom{A}{B_{T_1}}|$ and $\delta_2=|\det\binom{A}{B_{T_2}}|$.  By Cramer's rule, the entries of $
\bm x'$ are integer multiples of $\frac{1}{\delta_1}$ and the entries of $\bm x_{\min}$ are integer multiples of $\frac{1}{\delta_2}$.  Then, by letting $\delta'=\lcm(\delta_1,\delta_2)$, we have that the entries of $(\bm x' - \bm x_{\min})$ are integer multiples of $\frac{1}{\delta'}$.  Since $\bm c$ is an integer vector, we have that $\bm c^\T(\bm x' - \bm x_{\min})\geq\frac{1}{\delta'}$, and by the definition of $\delta$, we have that $\frac{1}{\delta'}\geq\frac{1}{\delta^2}$.  This is a contradiction to the fact that $\bm c^\T(\bm x'-\bm x_{\min})<1/\delta^2$.
\end{proof}

Note that the above proof also establishes that the result obtained by~\cite{DL15} regarding the number greatest-improvement augmentations needed to solve an equality form LP extends to the general-form LP (trivially by taking $\gamma=1$). Therefore, an easy corollary of the above result shows a (weakly) polynomial bound on the circuit-diameter of a rational polyhedron. We recall the statement of Corollary~\ref{cor:circuit_diameter}:

\newtheorem*{cordiam}{Corollary~\ref{cor:circuit_diameter}}
\begin{cordiam} 
There exists a polynomial function $f(m,\alpha)$ that bounds above the circuit-diameter
of any rational polyhedron $\mcp=\pol$ with $m$ row constraints and maximum encoding length among the coefficients in
its description equal to $\alpha$.
\end{cordiam}

\begin{proof}
Consider a polyhedron $\mcp = \pol$, and without loss of generality assume $m = m_{A} + m_{B} \geq n$. Recall that, 
by possibly scaling, we can assume that all coefficients in its description (i.e., all entries of $A,B,\bm{b}, \bm{d})$ are integers.
Let $\bm{x}$ and $\bm{\bar x}$ be two extreme points of $\mcp$. Let $T:=\{i : B_i \bm{\bar x} = \bm{d}(i)\}$.
Let $\bm c^\T$  be the vector obtained by adding the rows of $A$ and $B_T$, multiplied by $-1$.
By construction, using this vector as an objective function, $\bm{\bar x}$ is the unique optimal solution to the LP problem~(\ref{lp}).
Lemma~\ref{lem:approximation} shows that we can reach $\bm{\bar x}$ from $\bm{x}$ with 
$O(n\log\big( \delta \,\bm c^\T (\bm x-\bm {\bar x}) \big))$ augmentations. Note that $\bm c^\T (\bm x-\bm{ \bar x}) \leq ||\bm c||_{\infty} ||\bm x-\bm{ \bar x}||_{1}$. The result then follows by observing that $\log(||\bm c||_{\infty}) = O(\alpha + \log m)$,
$\log(||\bm x-\bm{ \bar x}||_{1} )\leq \log(2n \bar \delta)$ (using Cramer's rule), and $\log(\delta) \leq \log(\bar \delta)= O(n(\alpha+\log n))$.
\end{proof}

\subsection{Steepest-descent circuit-pivot rule}
\label{sec:steep_pr}

Using the  approximation result developed in the previous section, we here give a new bound on the number of steepest-descent augmentations needed to solve a bounded LP.  

\begin{thm}
\label{thm:steep_aprrox_deep}
Let $\omega_1$ denote the minimum 1-norm distance from any extreme point $\bm v \in \mcp$ to any facet $F$ of $\mcp$ such that $\bm v\notin F$.  Let
$M_1$ be the maximum 1-norm distance between any pair of extreme points of $\mcp$.
Using a steepest-descent circuit-pivot rule, a circuit-augmentation algorithm reaches an optimal solution $\bm x_{\min}$ of a bounded LP \emph{(\ref{lp})} from any initial feasible solution
$\bm x_0$, performing 
$$\mathcal{O}\left(n^2 \frac{M_1}{\omega_1} \,\log\big(\delta \,\bm c^\T (\bm x_0-\bm x_{\min})\big)\right)$$  augmentations. 
\end{thm}

\begin{proof} 
First, we can apply~\autoref{cor:snap_to_vert} to move from $\bm x_0$ to an extreme point $\bm x'$ of the LP in at most $n$ steps.

Let $\bm{\hat z}$ be an optimal solution to $Steep(\mcp,\bm x',\bm c)$, and let $\bm z:= \frac{1}{\|\bm{\hat z}\|_1} \bm{\hat z}$.
Note that $\bm z$ is a circuit of $\mcp$, being a rescaling of $\bm{\hat z} \in \C(A,B)$.
Let $\alpha\bm z$  be a steepest-descent augmentation at $\bm x'$. 
Similarly, let $\bm{\hat z^*}$ be an optimal solution to $Great(\mcp,\bm x',\bm c)$, let $\bm z^*:= \frac{1}{\|\bm{\hat z^*}\|_1} \bm{\hat z^*}$
and let $\alpha^*\bm z^*$ be a greatest-improvement augmentation at $\bm x'$.  
Then we have that
$-(\bm c^\T\bm{\hat z})/\|\bm{\hat z}\|_1 \geq -(\bm c^\T\bm{\hat z^*})/\|\bm{\hat z^*}\|_1$, and so $-\bm c^\T\bm z\geq -\bm c^\T\bm z^*$.  Therefore
$$
-\alpha\bm c^\T\bm z \geq -\alpha\bm c^\T\bm z^*=\left(\frac{\alpha}{\alpha^*}\right)(-\alpha^*\bm c^\T\bm z^*).
$$
Since the augmentation $\alpha\bm z$ is maximal, we have that at the point $\bm x'+\alpha\bm z$, there exists some facet of our feasible region which contains $\bm x'+\alpha\bm z$ but not $\bm x'$.  Then $\omega_1\leq \|(\bm x'+\alpha\bm z) - \bm x'\|_1 = \alpha\|\bm z\|_1$.  Since $\|\bm z\|_1=1$, it follows that $\alpha\geq\omega_1$.
Since $\bm x'+\alpha^*\bm z^*$ is feasible, we have that $\|(\bm x'+\alpha^*\bm z^*)-\bm x'\|_1$ is at most the maximum 1-norm distance from $\bm x'$ to any other feasible point.  As above, it follows that $\alpha^*$ is at most the maximum 1-norm distance from $\bm x'$ to any other feasible point.   
Since the function $f(\bm y)=\|\bm y-\bm x'\|_1$ is convex, 
this maximum is achieved at an extreme point.  It follows that $\alpha^*\leq M_1$.

\smallskip
Given these bounds on $\alpha$ and $\alpha^*$, it follows that
\begin{equation}
-\alpha\bm c^\T\bm z \geq \left(\frac{\omega_1}{M_1}\right)(-\alpha^*\bm c^\T\bm z^*) \tag{$\star$}
\label{eq_approx}
\end{equation}
Now let $\bar{\bm x}=\bm x'+\alpha\bm z$.  By~\autoref{cor:snap_to_vert}, an extreme point solution $\hat{\bm x}$ can be found from $\bar{\bm x}$ in at most $n-1$ additional augmentations (e.g., using again steepest-descent augmentations, but on a sequence of face-restricted LPs) with $\bm c^\T\hat{\bm x}\leq\bm c^\T\bar{\bm x}$.  Then we have that $\hat{\bm x}-\bar{\bm x}$ is an $\big(\frac{\omega_1}{M_1} \big)$-approximate greatest-improvement augmentation at $\bm x'$, and since $\hat{\bm x}$ is also an extreme point, we can continue to apply this procedure.  Since it takes at most $n$ steepest-descent augmentations to find such an $\big(\frac{\omega_1}{M_1} \big)$-approximate greatest-improvement augmentation, it follows from Lemma~\ref{lem:approximation} that from an initial solution $\bm x_0$, we can reach $\bm x_{\min}$ in $\mathcal{O} \left(n^2\frac{M_1}{\omega_1} \log\big(\delta \bm c^\T(\bm x_0-\bm x_{\min})\big)\right)$ steepest-descent augmentations. 
 \end{proof}

We note that the inequality~\eqref{eq_approx} yields the following corollary.

\begin{cor}
\label{cor_approx}
Let $\bm x$ be an extreme point of a bounded LP. A steepest-descent augmentation is a $\left(\frac{M_1}{\omega_1}\right)$-approximate greatest-improvement augmentation.
\end{cor}

The example given by~\autoref{fig:tight_approx_for_steep} shows that the approximation factor $\frac{M_1}{\omega_1}$  can be tight.
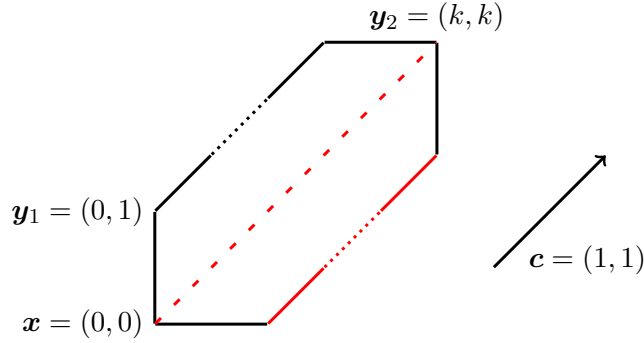
\begin{figure}
    \centering
    \begin{tikzpicture}[scale=1.5]
	    \tikzset{graph node/.style={shape=circle,draw=black,inner sep=0pt, minimum size=4pt}}
	    \tikzset{vert/.style={draw=none,fill=none,inner sep=0pt,minimum size=0pt}}
	    \tikzset{redge/.style={red, line width=.4mm}}		
	    \tikzset{dashredge/.style={red, line width=.4mm, dashed,dash pattern=on 3pt off 7pt}}
	    \tikzset{dotedge/.style={black, line width=.4mm, dotted}}
	    \tikzset{dotredge/.style={red, line width=.4mm, dotted}}
	    \tikzset{edge/.style={line width = .4mm}}
	    \tikzset{cedge/.style={->, line width=.4mm}}
	
	    \node[vert, label=left:{$\bm x = (0,0)$}] (x) at (0,0){};
	    \node[vert, label=left:{$\bm y_1 = (0,1)$}] (y1) at (0,1){};
	    \node[vert, label={$\bm y_{2} = (k,k)$}] (y2) at (2.5,2.5){};
	    \node[vert] (v1) at (2.5,1.5){};
	    \node[vert] (v2) at (1,0){};
	    \node[vert] (v3) at (1.5,2.5){};
	    
	    \node[vert] (c1) at (3,.5){};
	    \node[vert] (c2) at (4,1.5){};
	    
	    \node[vert] (i1) at (.5,1.5){};
	    \node[vert] (i2) at (1,2){};
	    
	    \node[vert] (j1) at (1.5,.5){};
	    \node[vert] (j2) at (2,1){};
	    
	    \draw[edge] (x)--(y1);
	    \draw[edge] (y1)--(i1);
	    \draw[dotedge] (i1)--(i2);
	    \draw[edge] (i2)--(v3);
	    \draw[edge] (v3)--(y2);
	    \draw[edge] (y2)--(v1);
	    \draw[redge] (j2)--(v1);
	    \draw[dotredge] (j1)--(j2);
	    \draw[redge] (j1)--(v2);
	    \draw[edge] (x)--(v2);
	    \draw[dashredge] (x)--(y2);
	    
	    \draw[cedge]  (c1)--(c2) node[pos=.3, below,xshift=8mm]{$\bm c=(1,1)$};
	\end{tikzpicture}
	\caption{This gives a family of examples (parameterized by $k$) where $i)$ moving along the edges incident at a vertex yields an arbitrarily bad approximation of moving along the greatest-improvement circuit, and $ii)$ a steepest-descent augmentation at $\bm x$ is a (tight) $\frac{M_1}{\omega_1}$-approximate greatest-improvement augmentation.  This polygon has vertices $\bm x=(0,0), \bm y_1=(0,1), \bm y_2=(k,k), (k,k-1), (k-1,k)$ and $(1,0)$.  One can check that at $\bm x$, $\bm y_1$ is both a steepest-descent augmentation as well as a steepest edge, while $\bm y_2$ is a greatest-improvement augmentation. We have that $\bm c^\T\bm y_1=\frac{1}{2k}\bm c^\T\bm y_2=\frac{\omega_1}{M_1}\bm c^\T\bm y_2$.}
	\label{fig:tight_approx_for_steep}
\end{figure}

\section{Implications for 0/1 Polytopes}\label{sec:01pol}

In this section, we explore the implications that~\autoref{thm:steep_aprrox_deep} has in the case of
0/1 polytopes, eventually proving Theorem~\ref{thm:approx_edge} and Theorem~\ref{thm:main1}. We start with the following lemma:

\begin{lem}\label{lem:0/1_steep_is_edge}
Consider a problem in the general form~\emph{(\ref{lp})} whose feasible region $\mcp$ is a 0/1 polytope, and let $\bm x$ be a non-optimal vertex of $\mcp$. 
Then, the optimal solution to $Steep(\mcp,\bm x,\bm c)$ corresponds to an edge-direction at $\bm x$, and can be computed in polynomial time.
\end{lem}
\begin{proof}
Consider the optimal objective function value of $Steep(\mcp,\bm x,\bm c)$. It is not difficult to see that this value is bounded above by the optimal objective function value of the following optimization problem $\mathcal Q$:

\begin{align}
    &\max  -\bm c^\T\bm z \nonumber\\
    &\text{s.t.} \nonumber\\
    &\|\bm z\|_1\leq1\label{eq:orthoplex}\\
    &\bm x+\varepsilon\bm z\in \mcp & \text{for some }\varepsilon>0\label{eq:feasible_cone}
\end{align}

over all $\bm z\in\R^n$.  This is true since if $\bm g \in \mathcal C(A,B)$ is a feasible solution to $Steep(\mcp,\bm x,\bm c)$, then 
$\frac{\bm g}{\|\bm g\|_1}$ is a feasible solution of $\mathcal Q$ with the same objective function value.

\smallskip
Let $\mcp_{\mathcal Q}$ denote the feasible region of $\mathcal Q$.
Note that $\mcp_{\mathcal Q}$ is the feasible cone at $\bm x$ in $\mcp$---given by the constraint~(\ref{eq:feasible_cone})---intersected with an $n$-dimensional cross-polytope---given by the constraint~(\ref{eq:orthoplex}).  The constraint~(\ref{eq:orthoplex}) can be modeled using the linear constraints
\begin{align}
    \bm v^\T\bm z\leq 1 &\,\,\text{ for all }\bm v\in\set{1,-1}^n.
\end{align}

It follows that $\mcp_{\mathcal Q}$ is a polytope, and therefore $\mathcal Q$ is a feasible bounded LP.  
There exists an optimal vertex $\bm y$ of $\mcp_{\mathcal Q}$ which is determined by $n$ linearly independent constraints of 
$\mcp_{\mathcal Q}$.

\smallskip
Since $\bm x\in\set{0,1}^n$ and $\mcp$ is a 0/1 polytope, each entry of $\bm x$ is either equal to its upper bound or its lower bound.  Thus, the feasible cone at $\bm x$ lies within a single orthant of $\R^n$.  This implies that among all the linear constraints that model $\|\bm z\|_1\leq 1$, only one is facet defining. 
Therefore, $\bm y$ is contained in at least $n-1$ facets corresponding to inequalities that describe the feasible cone at $\bm x$. Since $\bm x$ is not optimal, $\bm y\neq \bm 0$.

\smallskip
As a consequence of this, we have that the optimal solution of $\mathcal Q$ corresponds to an edge-direction of $\mcp$ incident with $\bm x$. It follows that the optimal solution of $Steep(\mcp,\bm x,\bm c)$ is an edge-direction of $\mcp$ incident with $\bm x$, and it can be computed in polynomial time.
\end{proof}

We note that it was shown in~\cite{Borgwardt+Viss} that for a slightly different definition of steepest-descent, a steepest-descent circuit-augmentation can be computed in polynomial time for all LPs.  Despite similarities in the two versions of the steepest-descent pivot rule, the technique they employ cannot be straightforwardly applied to work for the definition of steepest-descent used here, and our computability 
result in~\autoref{lem:0/1_steep_is_edge} is not implied by theirs.

We can now prove Theorem~\ref{thm:approx_edge}, which we restate  for convenience.

\newtheorem*{thm01}{Theorem~\ref{thm:approx_edge}}

\begin{thm01}
Let $\bm{x}$ be an extreme point of a 0/1-LP with $n$ variables. An augmentation along a steepest edge yields an $n$-approximation of an augmentation along a greatest-improvement circuit.
\end{thm01}

\begin{proof}
We combine Corollary~\ref{cor_approx} and Lemma~\ref{lem:0/1_steep_is_edge}. Since $\mcp$ is a 0/1 polytope, we immediately have that $M_1\leq n$.  We now show that $\omega_1\geq 1$.  Let $\bm v$ be any extreme point of $\mcp$ and let $F$ be any facet of $\mcp$ which does not contain $\bm v$.  By reflecting and translating $\mcp$, we may assume without loss of generality that $\bm v=\bm 0$ (Note that these operations do not change the 1-norm distance between any pair of points in $\mcp$).  It therefore suffices to show that for any facet $F$ not containing $\bm 0$, $\|\bm y\|_1\geq 1$ for all $\bm y\in F$.  Since all points in $F$ have non-negative coordinates, the minimum value of $\|\bm y\|_1$ over all $\bm y\in F$ is equal to the optimal solution to the following LP: 
\begin{equation*}
	\begin{aligned}
	    &\min \bm 1^\T\bm y\\
	    &\text{s.t.}\\
	    &\bm y\in F.
	\end{aligned}
\end{equation*}
There exists an optimal solution $\bm y^*$ to this LP which is an extreme point solution. Since $\bm y^*$ is an extreme point of $F$, it is also an extreme point of $\mcp$, and since $\bm y^*\in F$, it is an extreme point not equal to $\bm 0$.  Therefore, $\bm y^*$ has at least one coordinate equal to 1, and so $\|\bm y^*\|_1\geq 1$, as desired.  Therefore, $\frac{\omega_1}{M_1}\geq \frac{1}{n}$. 
\end{proof}

We will rely next on the following result of Frank and Tardos~\cite{andras+eva}.

\begin{lem}[\cite{andras+eva}]
\label{lem:replace_obj} Let $\bm w \in \mathbb R^n$ be a rational vector, and $\alpha$ be a positive integer.
Define $N:=(n+1)!2^{n\alpha}+1$. Then one can compute an integral vector $\bm w' \in \mathbb Z^n$ satisfying:
\begin{itemize} 
    \item[(a)] $\|\bm w'\|_{\infty}\leq2^{4n^3}N^{n(n+2)}$;
    \item[(b)] Consider any rational LP of the form $\max\set{\,\bm w^\T\bm x\ :\, A' \bm x \leq \bm b', \, \bm x\in\R^n\,}$, where 
    the encoding length of any entry of $A'$ is at most $\alpha$. Then,
$\bm x \in \mathbb R^n$ is an optimal solution to that LP  if and only if it is an optimal solution to $\max\set{\,\bm w'^\T\bm x\ :\, A' \bm x \leq \bm b', \, \bm x\in\R^n\,}$.  
\end{itemize}

 \end{lem}

We are now ready to prove the following theorem.

\begin{thm}
\label{thm:simplex_0/1}
Given a problem in the general form~\emph{(\ref{lp})} whose feasible region $\mcp$ is a 0/1 polytope, a circuit-augmentation algorithm with a steepest-descent circuit-pivot rule
 reaches an optimal solution performing
 a strongly-polynomial number of augmentations. \\
 Furthermore, if the initial solution is a vertex, 
 the algorithm follows a path on the 1-skeleton of $\mcp$.
\end{thm}

\begin{proof}
Let us call $(LP_1)$ the given LP problem of the form~(\ref{lp}) whose feasible region is $\mcp$. 
Since $\mcp$ is a $0/1$ polytope, for the sake of the analysis we can assume that the maximum absolute value of any element
in $A$ and $B$ is $\leq\frac{n^{n/2}}{2^n}$~\cite{Ziegler97}.
Apply Lemma~\ref{lem:replace_obj} to the LP obtained from $(LP1)$ after changing the objective function
to $\max \bm w^\T\bm x$, with $\bm w := - \bm c$. Set $\bm c' := - \bm w'$. 
Finally, let $(LP2):= \min\set{\,\bm c'^\T\bm x\ :\, A \bm x = \bm b, \, B \bm x \leq \bm d, \, \bm x\in\R^n\,}$.

\smallskip
Let $\bm x_0$ and $\bm x_{\min}$ be respectively the initial solution and the optimal solution.
By performing at most $n$ additional augmentations, we can assume $\bm x_0$  is an extreme point.

First, we will show that~\autoref{thm:simplex_0/1} holds for $(LP2)$.  Then, we will show that a circuit-augmentation algorithm traverses the same edge-walk when solving $(LP2)$ and $(LP1)$, if one uses the steepest-descent circuit-pivot rule. This will prove the statement.

\smallskip
Recall that the steepest-descent circuit-pivot rule selects at each step an improving circuit $\bm g$
that minimizes $\frac{\bm c^\T\bm g}{\|\bm g\|_{1}}$.  
Since the feasible region of $(LP2)$ is a $0/1$ polytope, we can apply Lemma~\ref{lem:0/1_steep_is_edge}. Therefore, 
each augmentation corresponds to moving from an extreme point to an adjacent extreme point. Furthermore, the total number of augmentations 
can be bounded via Theorem~\ref{thm:approx_edge} and Lemma~\ref{lem:approximation} by

$$
\mathcal{O}\left(n^2\,\log\big(\delta \,\bm c'^\T (\bm x_0-\bm x_{\min})\big)\right).
$$

Since the maximum absolute value of any entry of $A$ and $B$ is at most $\frac{n^{n/2}}{2^n}$, we have that the maximum absolute value of the determinant
of any $n \times n$ submatrix is at most $\big(\frac{n^{n/2}}{2^n} \big)^n n!$.
By the definition of $\delta$, we have that $\delta\leq \big( \big(\frac{n^{n/2}}{2^n} \big)^n n!\big)$, and so $\log(\delta)$ is polynomial in $n$.

\smallskip
Finally, we address the term $\log\big(\bm c'^\T(\bm x_0 - \bm x_{\min})\big)$.  Since $\bm x_0$ and $\bm x_{\min}$ are both in $\set{0,1}^n$, we have that $\log\big(\bm c'^\T(\bm x_0 - \bm x_{\min})\big)\leq\log(\|\bm c'\|_1)\leq \log(n\|\bm c'\|_{\infty})$, which is polynomial in $n$ due to Lemma~\ref{lem:replace_obj}(a).  Therefore, the number of augmentations required to solve $(LP2)$ is 
strongly-polynomial in the input size.

\smallskip
To finish our proof, it remains to show that when the circuit-augmentation algorithm is applied to $(LP1)$, it performs the same edge-walk as it does when it is applied to $(LP2)$. To see this, we will rely on the polyhedral characterization of the problem $Steep(\mcp,\bm x,\bm c)$, used in the proof of Lemma~\ref{lem:0/1_steep_is_edge}. As explained there, the edge-direction $\bm g$ selected by our algorithm applied to $(LP2)$ is an optimal solution to the LP $\max \{-\bm{c'}^\T \bm{x}: \; \bm{x} \in \mathcal P_{\mathcal Q}\}$, which describes $Steep(\mcp,\bm x,\bm c')$. Note that the maximum absolute value of a matrix-coefficient of this LP is also at most $\frac{n^{n/2}}{2^n}$. Therefore, due to Lemma~\ref{lem:replace_obj}(b), $\bm g$ is an optimal solution to $\max \{-\bm{c'}^\T \bm{x}: \; \bm{x} \in \mathcal P_{\mathcal Q}\}$ (i.e., $Steep(\mcp,\bm x,\bm c')$) if and only if it is an optimal solution to $\max \{-\bm{c}^\T \bm{x}: \; \bm{x} \in \mathcal P_{\mathcal Q}\}$ (i.e., $Steep(\mcp,\bm x,\bm c)$). Therefore, the circuit-augmentation algorithm implemented according to the steepest-descent circuit-pivot rule, performs the exact same pivots for the objective functions $\bm c'$ and $\bm c$. 
\end{proof}

 It can be readily seen that Theorem~\ref{thm:simplex_0/1} implies Part (i) of Theorem~\ref{thm:main1}. 
 This gives a strongly-polynomial bound on the number of \emph{distinct} basic feasible solutions visited with a steepest-descent pivoting rule. 
 In the context of the Simplex method, moving to a neighboring vertex along a steepest edge might require several \emph{degenerate} basis exchanges. Having a `pure' pivoting rule that implies a strongly-polynomial bound on the total number of basis exchanges remains an open question, but if the polytope is non-degenerate, then the two concepts coincide, hence we obtain Part (ii) of Theorem~\ref{thm:main1}.
 
 We conclude our paper with the following remark. As already mentioned,
in the context of pivot rules for the Simplex method, the
name \emph{steepest-descent} often refers to normalizations according to the $2$-norm of a vector, rather
than the $1$-norm. Since for any vector $\bm g \in \mathbb R^n$,   we
have  $\sqrt(n) ||\bm g||_2  \geq  ||\bm g||_1 \geq ||\bm g||_2$, it is not difficult to note that the result of Theorem~\ref{thm:simplex_0/1} still holds if we
normalize according to the $2$-norm.

\paragraph{Acknowledgements.} J.A. De Loera was partially supported by NSF grant DMS-1818969 and an 
NSF TRIPODS award (NSF grant no. CCF-1934568). S. Kafer and L. Sanit\`a were 
supported by the NSERC Discovery Grant Program and an Early Researcher Award by the Province of Ontario. 
We are grateful to the two anonymous referees who suggested several excellent recommendations and corrections. 
We are also grateful to Jon Lee, Thomas McCormick, and Francisco Barahona who provided useful references.

\bibliographystyle{plain}
\bibliography{biblioAugmentation.bib}

\end{document}